\documentclass[11pt]{amsart}
\usepackage{preamble,enumerate}
\usepackage{xcolor}

\definecolor{ultramarine}{RGB}{0,32,96}
\usepackage{enumitem}
\newcommand*{\defeq}{\mathrel{\vcenter{\baselineskip0.5ex
      \lineskiplimit0pt \hbox{\scriptsize.}\hbox{\scriptsize.}}}
  =}%

\def\Z{{\mathbb Z}}

\title{The Farrell-Jones conjecture for\\ hyperbolic-by-cyclic
  groups}%

\author{Mladen Bestvina}
\address{\tt M.\ Bestvina, Department of Mathematics, University of Utah, 155 S.\ 1400 E.\,
  \newline Salt Lake City, UT 84112, U.S.A.
  \newline http://www.math.utah.edu/$\sim$bestvina/} %
\email{\tt bestvina@math.utah.edu}
\author{Koji Fujiwara}
\address{\tt K.\ Fujiwara, Department of Mathematics, Kyoto University. Kyoto 606-8502, Japan.
  \newline http://www.math.kyoto-u.ac.jp/$\sim$kfujiwara/} %
\email{\tt kfujiwara@math.kyoto-u.ac.jp}
\author{Derrick Wigglesworth} %
\address{\tt D.\ Wigglesworth, Department of Mathematical Sciences, University
  of Arkansas, 309 SCEN, Fayetteville, AR 72703, U.S.A.
  \newline http://www.drwiggle.com} %
\email{\tt drwiggle@uark.edu}
\thanks{\today}

\begin{document}

\begin{abstract} We prove the Farrell-Jones Conjecture for mapping
  tori of automorphisms of virtually torsion-free hyperbolic
  groups. The proof uses recently developed geometric methods for
  establishing the Farrell-Jones Conjecture by Bartels-L\"uck-Reich, as well as the structure
  theory of mapping tori by Dahmani-Krishna.
\end{abstract}

\thanks{ The first author is supported by the NSF under the grant
  number DMS-1905720.  The second author is supported in part by
  Grant-in-Aid for Scientific Research (No.15H05739, 20H00114).  The third
  author would like to thank the Fields Institute for its
  hospitality. } 
\maketitle

\section{Introduction}
\label{sec:intro}

The Farrell-Jones Conjecture (which we will frequently abbreviate by
FJC) was formulated in \cite{FJ:FJC}.  For a group $G$, the
Farrell-Jones Conjecture (relative to the family $\vcyc$ of virtually
cyclic subgroups)
 predicts an isomorphism between the $K$-groups
(resp.\ $L$-groups) of the group ring $RG$ and the evaluation of a
homology theory on a certain type of classifying space for $G$.  Such
a computation, at least in principle, gives a way of classifying all
closed topological manifolds homotopy equivalent to a given closed
manifold of dimension $\geq 5$, as long as the FJC is known for the
fundamental group.  Particularly significant consequences include the
Borel conjecture (a homotopy equivalence between aspherical manifolds
can be deformed to a homeomorphism) and the vanishing of the Whitehead
group $Wh(G)$ for torsion-free $G$.

There are different versions of the Farrell-Jones conjecture; we will
always mean ``with coeffcients in an additive category with a 
$G$-action'' and relative to $\vcyc$, see e.g. \cite[Section 4.2]{BB:FJCMCG}. This version has
convenient inheritance properties, for example passing to subgroups
and finite index supergroups.

The Farrell-Jones Conjecture has generated much interest in the last
decade, due in no small part to the recent development of an axiomatic
formulation that both satisfies useful inheritance properties, and
provides a method for proving the conjecture.  For example, FJC is
currently known for hyperbolic groups \cite{BLR:EquivariantCovers},
relatively hyperbolic groups \cite{Bar:RelHypFJC}, CAT$(0)$ groups
\cite{Wegner:FCJ-cat0}, virtually solvable groups
\cite{Wegner:FJC-Virt-solvable}, $\text{GL}_n(\mathbb{Z})$
\cite{BLRR:FJC-GLn}, lattices in connected Lie groups \cite{KLR:FJCLatticesInLieGroups}, and
mapping class groups \cite{BB:FJCMCG}. The reader is invited to
consult the papers
\cite{BLR:FJC-Applications,Bar:RelHypFJC,Bar:OnProofsofFJC,BB:FJCMCG,BFL:jams}
for more information on applications of FJC and methods of proof.

The primary goal of this paper is the following theorem.
\begin{theorem}\label{thm:main-hyperbolic}
  Let $G$ be a virtually torsion-free hyperbolic group and $\Phi\colon G\to G$
  be an automorphism of $G$.  Then the hyperbolic-by-cyclic group
  $G_\Phi=G\rtimes_\Phi\Z$ satisfies the $K$- and $L$-theoretic Farrell-Jones
  conjectures.
\end{theorem}

Conjecturally, all hyperbolic groups are virtually torsion-free. In
our proof we quickly reduce to the case of torsion-free hyperbolic
groups, which simplifies the arguments. It is possible that a similar
proof can be given for general hyperbolic groups, but we decided
against attempting such a proof in view of the new technical difficulties
and the above conjecture.

Theorem \ref{thm:main-hyperbolic} generalizes the case when $G$ is a
finite rank free group, which we proved in \cite{BFW}, and the present
paper supersedes it, so \cite{BFW} will not be published
independently. Br\"uck, Kielak and Wu generalized this result to
infinitely generated free groups and further to normally poly-free
groups in \cite{BKW}.

Perhaps the most important consequence is the following:

\begin{theorem}\label{thm:extensions-hyperbolic}
  Let $G$ be a virtually torsion-free hyperbolic group and let
  \begin{equation*}
    1\longrightarrow G\longrightarrow H\longrightarrow Q\longrightarrow 1
  \end{equation*}
  be a short exact sequence of groups.  If $Q$ satisfies the
  $K$-theoretic (resp.\ $L$-theoretic) Farrell-Jones Conjecture, then
  so does $H$.
\end{theorem}

The group $G_\Phi$ appearing in the statement of Theorem
\ref{thm:main-hyperbolic} depends only on the outer class of $\Phi$,
which we denote by $\phi$, and sometimes we use $G_\phi$ for $G_\Phi$.
Outer automorphisms of torsion free hyperbolic groups are well
understood (see
\cite{sela,Levitt:AutomorphismsHyperbolicGroups,GL:RelativeOuterSpace}
for a small sample).  Broadly speaking, the tools used to study such
automorphisms differ based on the number of ends of $G$: when $G$ is
one-ended, techniques from JSJ decompositions are used; when $G$ has
infinitely many ends, one uses the Grushko decomposition and the
associated relative Outer space. Our proof of Theorem
\ref{thm:main-hyperbolic} is no different; though our techniques in
the latter case are further divided according to whether the
automorphism $\phi$ is \emph{polynomially growing} or
\emph{exponentially growing}.

The structure of this paper is as follows: Section
\ref{sec:background} contains the background material on the
Farrell-Jones Conjecture necessary for our purposes.  In particular,
we review the geometric group theoretic means for proving FJC
developed in \cite{FJ:FJC} and subsequently refined and generalized in
\cite{BLR:EquivariantCovers,BLR:FJC-Hyperbolic}.  In Section
\ref{sec:one-ended} we treat the case that $G$ is a one-ended
hyperbolic group.  Then in Section \ref{sec:PG-autos}, we treat
polynomially growing automorphisms of groups with infinitely many
ends.  The case of exponentially growing automorpisms is easy in
comparison, thanks to the recent structure theorem of Dahmani-Krishna
\cite{DK:RelHypHypAutos} which asserts that $G_\Phi$ is relatively
hyperbolic; it is handled in Section \ref{sec:EG-autos}. In Section
\ref{sec:extensions} we prove Theorem \ref{thm:extensions-hyperbolic}
using standard techniques. The last section contains a proof of a
technical result needed for the polynomially growing case.

\section{Background}
\label{sec:background}
For the remainder of this paper, $G$ will denote a finitely generated
group.  While the Farrell-Jones Conjecture originated in
\cite{FJ:FJC}, its present form with coefficients in an additive
category is due to Bartels-Reich \cite{BR:FJC-Coeffs} (in the
$K$-theory case) and Bartels-L\"uck \cite{BL:FJC-Twisted} (in the
$L$-theory case). We
restrict ourselves to the axiomatic formulation.

A \emph{family} $\matheurm{F}$ of subgroups of $G$ is a non-empty
collection of subgroups which is closed under conjugation, taking
subgroups and finite index supergroups.
For example, the collection $\fin$ of finite subgroups of
$G$ is a family, as is the collection $\vcyc$ of virtually cyclic
subgroups of $G$.

\subsection{Geometric Axiomatization of FJC}
\label{Sec:GeometricMethods}

We first recall a regularity condition that has been useful in recent
results on the Farrell-Jones conjecture
\cite{BLR:FJC-Hyperbolic,Bar:RelHypFJC,Knopf:FJCTrees}.  Let $X$ be a
space on which $G$ acts by homeomorphisms and $\matheurm{F}$ be a
family of subgroups of $G$.  An open subset $U\subseteq X$ is said to
be an $\matheurm{F}$-subset if there is $F\in\matheurm{F}$ such that
$gU = U$ for $g\in F$ and $gU\cap U = \emptyset$ if $g\notin F$.  An
open cover $\mathcal{U}$ of $X$ is \emph{$G$-invariant} if
$gU\in\mathcal{U}$ for all $g\in G$ and all $U\in\mathcal{U}$.  A
$G$-invariant cover $\mathcal{U}$ of $X$ is said to be an
\emph{$\matheurm{F}$-cover} if the members of $\mathcal{U}$ are all
$\matheurm{F}$-subsets.  The \emph{order} (or multiplicity) of a cover
$\mathcal{U}$ of $X$ is less than or equal to $N$ if each $x\in X$ is
contained in at most $N + 1$ members of $\mathcal{U}$.

\begin{definition}
  Let $\matheurm{F}$ be a family of subgroups of $G$.  An action
  $G\curvearrowright X$ is said to be
  \emph{$N$-$\matheurm{F}$-amenable} if for any finite subset $S$ of
  $G$ there exists an open $\matheurm{F}$-cover $\mathcal{U}$ of
  $G\times X$ (equipped with the diagonal $G$-action) with the
  following properties:
  \begin{itemize}
  \item the multiplicity of $\mathcal{U}$ is at most $N$;
  \item for all $x\in X$ there is $U \in \mathcal{U}$ with
    $S\times \{x\}\subseteq U$.
  \end{itemize}
  An action that is $N$-$\matheurm{F}$-amenable for some $N$ is said
  to be \emph{finitely $\matheurm{F}$-amenable}.  We remark that such
  covers have been called {\it wide} in some of the literature.
\end{definition}

\subsection{The class $\AC(\vnil)$}
\label{sec:inheritance}
Following \cite{BB:FJCMCG} we now define the class of groups
$\AC(\vnil)$ that satisfy suitable inheritance properties and all
satisfy FJC.  Let $\vnil$ denote the class of finitely generated
virtually nilpotent groups.
Set $\ac^0(\vnil)=\vnil$ and
inductively $\ac^{n+1}(\vnil)$ consists of groups $G$ that admit a
finitely $\matheurm{F}$-amenable action on a compact Euclidean retract
(ER) with all groups in $\matheurm{F}$ belonging to
$\ac^n(\vnil)$. The action on a point shows that
$\ac^n(\vnil)\subseteq \ac^{n+1}(\vnil)$ and we set
$\AC(\vnil)=\bigcup_{n=0}^\infty \ac^n(\vnil)$.

\begin{proposition}[\cite{BB:FJCMCG}]\label{prop:FJCInheritance}
  \begin{enumerate}
  \item [(i)] $\AC(\vnil)$ is closed under taking subgroups, taking
    finite index supergroups, and finite products.
  \item [(ii)] All groups in $\AC(\vnil)$ satisfy the Farrell-Jones
    Conjecture.
  \end{enumerate}
\end{proposition}

Our main result can now be stated as follows.

\begin{theorem}\label{thm:main}
  Let $G$ be a torsion-free hyperbolic group and let $\Phi:G\to G$ be
  an automorphism.  Then $G_\Phi=G\ltimes_\Phi \Z$ belongs to
  $\AC(\vnil)$.
\end{theorem}

If $G$ is only virtually torsion-free, there is a finite index
characteristic subgroup $H<G$ which is torsion-free, and then $H_\Phi$
has finite index in $G_\Phi$, so Theorem \ref{thm:main-hyperbolic}
follows from Theorem \ref{thm:main} and Proposition
\ref{prop:FJCInheritance}(i). 

The following two theorems of Knopf and of Bartels will be
crucial. Recall that an isometric action of a group $G$ on a tree $T$
is {\it acylindrical} if there exists $D>0$ such that whenever $x,y\in
T$ are at distance $\geq D$ then the stabilizer of $[x,y]$ has
cardinality $\leq D$.

\begin{theorem}[{\cite[Corollary 4.2]{Knopf:FJCTrees}}]\label{Knopf}
  Let $G$ act acylindrically on a simplicial tree $T$ with finitely
  many orbits of edges. If all vertex stabilizers belong to
  $\AC(\vnil)$ then so does $G$.
\end{theorem}

\begin{theorem}[\cite{Bar:RelHypFJC}]\label{Bartels}
  Suppose $G$ is hyperbolic relative to a finite collection of
  subgroups, each of which is in $\AC(\vnil)$. Then $G$ also belongs
  to $\AC(\vnil)$.
\end{theorem}




\section{One-ended groups $G$}
\label{sec:one-ended}

This section is dedicated to proving our main result for one-ended
groups using JSJ decompositions and an analysis of the action of
$G_\phi$ on the associated Bass-Serre tree.  Throughout this section,
$G$ will denote a torsion-free one-ended Gromov hyperbolic group.

\begin{proposition}\label{prop:FJC1EndedHyperbolicByCyclic}
  Let $G$ be a one-ended torsion-free hyperbolic group and
  $\Phi\in\Aut(G)$.  Then the hyperbolic-by-cyclic group $G_\Phi$ is
  in $\AC(\vnil)$.
\end{proposition}

\subsection{JSJ decompositions}
The group $G$ (after Rips and Sela
\cite{RS:JSJHyperbolic,Sela:JSJHyperbolicII}, see also \cite{bowditch,dunwoody-sageev,fujiwara-papasoglou,guirardel-levitt}) has a JSJ
decomposition, $\Lambda$, which is a finite graph of groups such that
$\pi_1(\Lambda)=G$, each edge group is cyclic, and each vertex group
is of one of the following types:
\begin{itemize}
\item cyclic,
\item quadratically hanging (or QH), or
\item rigid.
\end{itemize}

QH vertex groups are represented by compact surfaces that carry a
pseudo-Anosov homeomorphism (pA) and whose boundary components exactly
represent the incident edge groups. Rigid vertex groups are non-cyclic
quasi-convex subgroups. A degenerate case occurs when $G$ is a closed
surface group (that carries a pA); in that case the JSJ decomposition
is a single QH vertex. The key property of JSJ decompositions is that
they are ``maximal'' such decompositions, and we will need the
following manifestation of it. Let $Out_\Lambda(G)$ be the group of
``visible'' automorphisms, i.e. the subgroup of $Out(G)$ generated by
Dehn twists in edge groups and in 1-sided simple closed curves in
surfaces representing the QH vertices.

\begin{theorem}[\cite{sela}]
  $Out_\Lambda(G)$ has finite index in $Out(G)$.
\end{theorem}

An important special case, when $G$ has no splittings over $\Z$, was
proved earlier by Paulin \cite{paulin}.

If we let $T$ be the Bass-Serre tree associated with $\Lambda$, then
for every Dehn twist above, and hence for every element $\phi\in
Out_\Lambda(G)$, and for every lift $\Phi\in Aut(G)$, there is an
isometry $t:T\to T$ which is $\Phi$-equivariant. It follows that the
group $G_\Phi$ also acts by isometries on the same tree $T$. This
action may not be acylindrical, and our proof consists of modifying
the JSJ decomposition in order to construct an acylindrical action and
apply Theorem \ref{Knopf}. All our modifications will have the
property that $t$ and $G_\Phi$ continue to act by isometries, and we
retain the name $t$.

Since $G_{\phi^k}$ has index $k$ in $G_\phi$ when $k>0$, inheritance to finite
index supergroups implies that we are free to replace $\phi$ by a power
when proving Theorem \ref{thm:main}.
To begin with, we
are assuming, by passing to a power, that
\begin{itemize}
  \item $\phi\in Out_\Lambda(G)$, and
    \item the restriction of $\phi$ to
any QH-vertex group is represented by a mapping class of the punctured
surface that fixes a (possibly empty) multicurve and in each
complementary component it is either identity or pseudo-Anosov.
\end{itemize}
\vskip 0.5cm
{\bf Modification 1.} We refine $\Lambda$ by replacing each QH-vertex
group with the graph of groups dual to the reducing multicurve in the
corresponding surface. The resulting graph of groups representing $G$
will be called $\Lambda_1$. Thus the number of edges of $\Lambda_1$ is
equal to the number of edges of $\Lambda$ plus the sum of the numbers
of curves in reducing multicurves in the QH vertex groups, and each QH
vertex group is replaced by a collection of vertex groups, one for
each complementary component of the reducing multicurve. It will now
be convenient not to distinguish between rigid vertex groups and
noncyclic vertex groups coming from complementary components where $\phi$ is
identity. We will use the following classification of vertex groups in
$\Lambda_1$. Let $\Phi\in Aut(G)$ be a representative of $\phi\in Out(G)$.
\begin{enumerate}
  \item [(i)] $R$-vertices: associated groups are noncyclic, and the
    restriction of $\Phi$ is a conjugation by an element of $G$,
    \item [(ii)] $pA$-vertices: these are represented by a surface with edge
      groups corresponding exactly to the boundary components, and the
      restriction of $\Phi$ is represented by a pseudo-Anosov mapping
      class, composed with conjugation,
      \item [(iii)] $Z$-vertices: these are cyclic, and the restriction of
        $\Phi$ acts by conjugation.
\end{enumerate}

Let $T_1$ be the Bass-Serre tree corresponding to $\Lambda_1$.
\vskip 0.5cm
{\bf Modification 2.} Here we arrange that distinct edges incident to
the same $R$-vertex in $T_1$ have non-commensurable stabilizers. Note that this
is already true for $pA$-vertices. Given an $R$-vertex, consider the
set of incident edges whose stabilizers are in a fixed
commensurability class and fold together thirds of these edges
incident to the vertex (only a third is to make sure that if folding from
both ends there is no intersection). Do this for all commensurability
classes. The new vertices created by folding are declared to be
$Z$-vertices. Call the resulting tree $T_2$. The edge stabilizers are still
cyclic, and (i)-(iii) still hold. In addition we have

\begin{enumerate}
  \item [(iv)] Distinct edges incident to a vertex of type $R$ or $pA$
    have non-com\-men\-su\-rable stabilizers.
\end{enumerate}

Before the next modification, we take a closer look at the
$Z$-vertices.

\begin{lemma}
  Let $E$ be a maximal cyclic subgroup of $G$. The set of edges of
  $T_2$ whose stabilizer is a subgroup of $E$ is finite and the union
  of these edges is a subtree $Z(E)$ of $T_2$. All vertices of $Z(E)$
  of valence $>1$ are of $Z$-type and 
  the tree $Z(E)$ intersects the closure of
  $T_2\smallsetminus Z(E)$ in the set of vertices of $R$- and
  $pA$-type, all of which are of valence 1 in $Z(E)$.
\end{lemma}

\begin{proof}
  If a nontrivial subgroup of $E$ fixes two edges in an edge-path,
  then it fixes all edges between them, so the claim that $Z(E)$ is a
  tree follows. Any vertex of $Z(E)$ of valence $>1$ must be of
  $Z$-type by (iv) and therefore of finite valence. It remains to show
  that $Z(E)$ has finite diameter. We will show that it does not
  exceed $2N$, where $N$ is the number of edges in the quotient graph
  $\Lambda_2=T_2/G$. 
  Indeed, assume there is a reduced edge path of $2N+1$ edges in
  $Z(E)$. Then there are two oriented edges $a,b$ that map to the same edge in
  $\Lambda_2$ with the same orientation, i.e. there is $g\in G$ with
  $g(a)=b$ and so that $a$ points to $b$ while $b$ points away from
  $a$. This implies that $g$ is a hyperbolic isometry of $T_2$, and it
  conjugates $E$ to $E$. Since maximal cyclic subgroups of $G$ are
  malnormal and $g\notin E$ this is a contradiction.
\end{proof}

We say that two $R$-vertices are {\it equivalent} if $\Phi:G\to G$
conjugates each of them by the same $g\in G$ (which is unique since
the $R$-vertex groups are not cyclic). Note that the action of $G$
preserves the equivalence classes, as does the $\Phi$-equivariant
isometry $t:T_2\to T_2$ induced by $t:T\to T$. 
\vskip 0.5cm
{\bf Modification 3.} Here we arrange that in each $Z(E)$
the convex hulls of equivalent $R$-vertices are pairwise disjoint. Say
$V$ is the set of $R$-type vertices in $Z(E)$ and $W$ is the set of
$pA$-type vertices in $Z(E)$. All of them have valence 1. Write
$V=V_1\sqcup V_2\sqcup \cdots\sqcup V_k$ by
grouping equivalent vertices into one $V_i$. Remove the interior
$Z(E)\smallsetminus (V\cup W)$ of $Z(E)$. Then for each $V_i$ introduce a vertex
$v_i$ and an edge joining $v_i$ to each vertex in $V_i$. Finally
introduce a new vertex $v$ and an edge joining $v$ to each $v_i$ and
to each $pA$-type vertex. Perform the same modification on all trees
$Z(E)$, so $G$ continues to act. The stabilizer of the vertex $v$
above is $E$ and the edge stabilizers are subgroups of $E$. All of the
newly introduced vertices are of type $Z$. Call the resulting tree
$T_3$. We retain the previous notation, so now in addition to (i)-(iv)
we also have
\begin{enumerate}
  \item [(v)] $Z(E)$ is a tree of
diameter $\leq 4$ and the convex hulls of equivalent $R$-vertices in
$Z(E)$ are pairwise disjoint.
\end{enumerate}

Before the final modification, we make some observations. First,
$\Phi$ takes the stabilizers of vertex and edge groups of $T_3$ to
other such stabilizers and induces an isometry $t:T_3\to T_3$, which is
type preserving. Thus $$G_\Phi=\langle G,t\mid tgt^{-1}=\Phi(g)\rangle$$ acts on
$T_3$. This action may not be acylindrical; for example the edges
whose stabilizer is fixed by $\Phi$ will be fixed by $t$, so we have
to ensure that such edges form a bounded set.

For $g\in G$ let $R(g)$ be the convex hull of the set of type $R$
vertices whose stabilizers are conjugated by $g$.

\begin{lemma}
  There are no $pA$-vertices in $R(g)$ and all $R$-vertices in $R(g)$
  are equivalent, so $\Phi$ conjugates all vertex and edge groups in
  $R(g)$ by $g$. For $g\neq g'$ the trees $R(g)$ and $R(g')$ are
  disjoint. 
\end{lemma}

\begin{proof}
  We may assume $g=1$ by rechoosing $\Phi\in Aut(G)$.
  Let $V,W$ be two $R$-vertex groups fixed by $\Phi$ elementwise and consider
  the edge-path $[v,w]$ joining the two vertices $v,w$. Thus $t$ fixes $[v,w]$
  and $\Phi$ leaves all vertex groups along $[v,w]$ invariant. Let $U$
  be such a vertex group corresponding to a vertex $u\in [v,w]$,
  $u\neq v,w$. If $U$ is cyclic, then a finite index
  subgroup appears as an edge group, and so it is fixed elementwise by
  $\Phi$. If $U$ is of $R$-type, $\Phi|U$ is conjugation by some
  element $h\in G$, but since the two incident edge groups along
  $[v,w]$ are non-commensurable and are both fixed by $\Phi$, we must
  have $h=1$, i.e. $U$ is fixed elementwise. By a similar argument,
  $U$ cannot be of $pA$-type (if $\Phi$ fixes the cyclic group
  corresponding to one boundary component, it must nontrivially
  conjugate all the others).

  Now suppose that $g'\neq 1$ and consider $R(1)\cap R(g')$. This
  intersection is a tree, and every vertex in it must be of
  $Z$-type. It follows from (v) that the intersection is empty.
\end{proof}
\vskip 0.5cm
{\bf Modification 4.} Collapse all $R(g)$'s to get a tree $T_4$. Since
$t:T_3\to T_3$ sends $R(g)$ to $R(\Phi(g))$,
it induces an isometry, still called $t:T_4\to T_4$. Thus $G_\Phi$
acts on $T_4$, and the new vertices corresponding to $R(g)$ are of $R$-type.

\begin{proposition}\label{acylindrical}
  All vertex and
  edge stabilizers of the action of $G_\Phi$ on $T_4$ are in
  $\AC(\vnil)$, and the action is acylindrical. 
\end{proposition}

\begin{proof}
  The edge stabilizers are $\Z\times\Z$, hence in $\AC(\vnil)$. The
  $G$-tree $T_4$ is minimal with cyclic edge stabilizers. It follows
  that vertex groups are finitely generated and quasi-convex in $G$,
  hence hyperbolic. There are 3 types of vertices in the $G_\Phi$-tree
  $T_4$, corresponding to the vertex types of the $G$-tree. If $v$ is
  a $Z$-type vertex in the $G$-tree, the corresponding vertex
  stabilizer in $G_\Phi$ is $\Z\times\Z$. If $v$ is an $R$-type vertex
  with stabilizer $V$, the corresponding stabilizer in $G_\Phi$ is
  $V\times\Z$, which belongs to $\AC(\vnil)$ by Proposition
  \ref{prop:FJCInheritance}(i).  If $v$ is of $pA$-type, the
  stabilizer in $G_\phi$ is the semi-direct product $V\rtimes_\phi\Z$,
  which belongs to $\AC(\vnil)$ (it is the fundamental group of a
  finite volume hyperbolic 3-manifold, so the statement follows from
  Theorem \ref{Bartels}, see also
  Roushon \cite{Roushon:FJC3Manifolds}).

  It remains to argue that the action of $G_\Phi$ is acylindrical,
  that is, we need a uniform bound on the diameter of $Fix(\gamma)$
  for every nontrivial $\gamma\in G_\Phi$. We consider different cases.

  {\it Case 1: $\gamma=g\in G$.} Then $Fix(\gamma)$ is empty, a point,
  or contained in
  $Z(E)$, where $E$ is the maximal cyclic subgroup of $G$ that
  contains $\gamma$, so it has diameter $\leq 4$ by (v).

  {\it Case 2: $\gamma=t$.} Then $Fix(\gamma)$, if not a single
  vertex, is the union 
  of edges whose stabilizer is fixed by $\Phi$. If $R(1)=\emptyset$,
  i.e. if $T_3$ has no $R$-vertices whose stabilizer is fixed by
  $\Phi$, then $Fix(t)$ is contained in a single $Z(E)$ so has
  diameter $\leq 4$ (otherwise the vertex in the intersection of two
  adjacent $Z(E)$'s containing edges of $Fix(t)$ would be an
  $R$-vertex whose stabilizer is elementwise fixed by
  $\Phi$). Otherwise, by the same argument, every such edge is
  contained in some $Z(E)$ that contains an $R$-vertex in $R(1)$, and
  so in $T_4$ the set $Fix(t)$ is contained in the 4-neighborhood of
  the vertex $R(1)$, so it has diameter $\leq 8$.

  {\it Case 3: $\gamma=ht$ for $h\in G$.} Then $\gamma$ fixes an edge
  $e$ if and only if $\Phi(g)=h^{-1}gh$ for $g\in Stab_G(e)$. In other
  words, if we replace $\Phi$ by the conjugate $h\Phi h^{-1}$ this
  case reduces to Case 2.

  {\it Case 4: $\gamma=t^k$ for $k\neq 0$.} If we replace $\Phi$ with
  $\Phi^k$ and repeat the modifications 1-4 to the JSJ-tree, the
  resulting tree is exactly the same. For example, the equivalence
  relation on the set of $R$-vertices doesn't change, since $g\neq h$
  implies $g^k\neq h^k$ in torsion-free hyperbolic groups. (This is one
  place where torsion-free hypothesis is used in a substantial way.)
  Thus $Fix(t^k)$ has diameter $\leq 8$.

  {\it Case 5: $\gamma=ht^k$ for $k\neq 0$, $h\in G$.} This is the
  general case, and it follows by combining Cases 3 and 4, that is,
  applying the argument to a different representative of $\phi^k$.
\end{proof}

\begin{proof}[Proof of Proposition
    \ref{prop:FJC1EndedHyperbolicByCyclic}]
  Follows from Theorem \ref{Knopf} applied to $G_\phi$ acting on
  $T_4$, using Proposition \ref{acylindrical}.
\end{proof}

We also record the following consequence, which may be of independent
interest. 

\begin{corollary}\label{fixed}
  Let $G$ be a 1-ended torsion-free hyperbolic group and $\phi\in
  Out(G)$. Then there is a power $\phi^N$, $N>0$, so that for every
  lift $\Phi\in Aut(G)$ of $\phi^N$ every $\Phi$-periodic element of
  $G$ is fixed (i.e. $\Phi^k(g)=g$, $k>0$ imply $\Phi(g)=g$).
\end{corollary}

\begin{proof}
  The power $\phi^N$ is the one we take in order to make the
  Modifications 1-4 above. We may assume $g\neq 1$, and consider
  $T_4$, where we set $t$ to be the isometry induced by $\Phi$. If $g$
  is hyperbolic then $t^k$ leaves the axis of $g$ invariant. If $t^k$
  is elliptic, it fixes the whole axis contradicting acylindricity. If
  $t$ and $t^k$ are loxodromic, then
  for
  suitable $m,s\neq 0$ the isometries $(t^k)^m$ and $g^s$ have the
  same action on this axis. This implies that $t^{km}g^{-s}$ fixes the
  axis, again contradicting
  acylindricity. Thus $g$ is elliptic. If $g$ fixes an edge or a
  vertex of the form $R(h)$, then $\Phi$ acts on $g$ by conjugation,
  so either $g$ is fixed or it is not periodic (this uses that $G$ is
  torsion-free). If $g$ fixes a $pA$ vertex but no edges, then $g$ is
  represented by a nonperipheral curve in a surface and $\phi^N$ acts
  as a pseudo-Anosov. Thus even the conjugacy class of $g$ is
  non-periodic.
\end{proof}

\section{The PG case ($\infty$-ly many ends)}
\label{sec:PG-autos}

We continue with the standing assumption that $G$ is a torsion-free
Gromov hyperbolic group.  The goal of the present section is the
following proposition, whose proof closely mirrors the analogous
result in the free group case, treated in \cite{BFW}.

\begin{proposition}\label{prop:PGFJC}
  Let $G$ be a torsion-free hyperbolic group with infinitely many ends
  and assume $\phi\in\Out(G)$ has polynomial growth.  
  Then the group $G_\phi$ belongs to
  $\AC(\vnil)$.
\end{proposition}

\subsection{Polynomial growth}
We start by reviewing the notion of polynomial growth in this setting;
for more details see \cite{DK:RelHypHypAutos}.

Let $G=G_1*G_2*\cdots*G_n*F_r$ be the Grushko free product
decomposition of $G$, where the $G_i$ are 1-ended hyperbolic groups
and $F_r$ is the free group of rank $r$ (we allow $n=0$ or $r=0$). The
{\it Grushko rank} of $G$ is $n+r$.
By
a {\it Grushko tree} we mean a minimal simplicial $G$-tree where the
edge stabilizers are trivial and the vertex stabilizers are precisely
the conjugates of the $G_i$ (and possibly the trivial group). Fix a
Grushko tree $T$. When $g\in G$ define $$||g||=\min d_T(v,g(v))$$ This
is a conjugacy invariant and it is a measure of the length of a
conjugacy class relative to the vertex stabilizers.  Let $\phi\in
Out(G)$. Then $\phi$ permutes the vertex stabilizers of any Grushko
tree. We say that $g\in G$ {\it grows polynomially} under $\phi$ if there
exists a polynomial $P_g$ such that $$||\phi^n(g)||\leq P_g(n)$$
We say that $\phi$ has {\it
  polynomial growth} if every $g\in G$ grows polynomially. These
definitions do not depend on the choice of the Grushko tree.

In \cite{collins-turner}, generalizing the case of free groups in
\cite{BH:TrainTracks}, Collins and Turner proved the following train
track theorem. For refinements and generalizations see
\cite{feighn-handel,francaviglia-martino,lyman}.

\begin{theorem}[\cite{collins-turner}]\label{CT}
  Let $G$ be as above and let $\phi\in Out(G)$ be a polynomially growing
  automorphism and $\Phi$ a lift of $\phi$ to $Aut(G)$. Then there is
  a Grushko tree $T$ and a $\Phi$-equivariant map $f:T\to T$ that
  sends vertices to vertices and such that
  \begin{enumerate}
  \item [(1)] there is a $G$-invariant filtration
    $$T_0\subset T_1\subset\cdots\subset T_{k-1}\subset T_k=T$$
    of $T$ into subgraphs (these are forests and need not be
    connected), with $T_0$ consisting of all vertices with nontrivial
    stabilizer, and
    \item [(2)] for every edge $e\in T_j$, $j=1,2,\cdots,k$, $f(e)$ is
      an edge-path that crosses exactly one edge in $T_j\smallsetminus
      T_{j-1}$ while all the other edges are in $T_{j-1}$.
  \end{enumerate}
\end{theorem}

It is useful to visualize what this is saying about the quotient graph
of groups $T/G$. There is a filtration by subgraphs starting with the
vertices representing the $G_i$'s, all these subgraphs are invariant
under the induced map $\overline f:T/G\to T/G$, and the image of each
edge in the stratum $T_{j}/G\smallsetminus T_{j-1}/G$ crosses one edge
of that stratum and the rest of the image is in $T_{j-1}/G$. In other words,
the transition matrix of $\overline f$ is an upper triangular block
matrix with all diagonal blocks permutation matrices. 

The Collins-Turner theorem is more general and applies to all
automorphisms, in general there are also exponentially growing and
zero strata, but they don't arise in the polynomially growing case.

Since we are allowed to take powers of $\phi$, the statement becomes a
bit simpler. The permutation matrices can be taken to be $1\times 1$.

\begin{proposition}\label{CT2}
  Let $G$ be as above and $\phi\in Out(G)$ polynomially growing. Then
  there is a power $\phi^N$ such that for every lift $\Phi$ of
  $\phi^N$ the following holds.
  \begin{enumerate}
    \item [(i)] Up to conjugacy, $\Phi$ preserves each $G_i$. 
      \item [(ii)] There is a Grushko tree $T$ and $f:T\to T$
        satisfying the conclusions of Theorem \ref{CT} for $\Phi$ such that, in
        addition, each stratum $T_j\smallsetminus T_{j-1}$ consists of
        exactly one orbit of (oriented) edges, and $f$ sends each such
        edge to an edge-path that crosses an edge of
        $T_j\smallsetminus T_{j-1}$ with the same orientation.
              \end{enumerate}
\end{proposition}

We say that an automorphism $\Phi:G\to G$ is {\it neat} if
$\Phi^k(g)=g$ for $k\neq 0, g\in G$ implies $\Phi(g)=g$. An outer
automorphism $\phi\in Out(G)$ is {\it neat} if every lift of $\phi$ to
$Aut(G)$ is neat.

We shall need the following theorem, whose proof is deferred to
Section \ref{s:neat}.

\begin{theorem}\label{neat}
  Let $G$ be a torsion free hyperbolic group and let $\phi\in Out(G)$
  be a polynomially growing outer automorphism (with respect to the
  Grushko tree). Then there is $N>0$ such that $\phi^N$ is neat.
\end{theorem}

The proof will show that $N$ depends only on $G$, not on $\phi$.

\begin{proof}[Proof of Proposition \ref{prop:PGFJC}]
  We will induct on the Grushko rank. The base of the induction is
  Proposition \ref{prop:FJC1EndedHyperbolicByCyclic}. Replace $\phi$
  by a power
  $\phi^N$ so that a lift $\Phi$ of $\phi^N$ is realized as in
  Proposition 
  \ref{CT2}, and so that $\Phi$ is neat. The highest edge in the
  filtration defines a splitting of $G$ as $G_1*G_2$ or as $G'*_1$,
  with $G_1,G_2,G'$ of strictly smaller Grushko rank and $\phi$
  inducing polynomially growing outer automorphisms on these
  factors, so by induction their mapping tori are in $\AC(\vnil)$.

We will treat the case where $G=G_1*G_2$; the argument requires only
notational changes
when $G$ is an HNN extension.  We can geometrically realize
$\phi$ as follows.  For $i\in\{1,2\}$, let $X_i$ be a presentation
complex for $G_i$: that is a 2-dimensional CW complex with
$\pi_1(X_i)=G_i$.  The automorphism $\phi|_{G_i}$ then tells us how to
build a homotopy equivalence $f_i\colon X_i\to X_i$ such that
$(f_i)_*=\phi|_{G_i}$.  Attaching $X_1$ and $X_2$ with an edge $e$
yields a graph of spaces $X$ with $\pi_1(X)=G$.  
We
can realize $\phi$ as a homotopy equivalence $f\colon X\to X$
restricting to $f_i$ on $X_i$ and such that $f(e)=\alpha_1 e \alpha_2$
for $\alpha_i$ a path in $X_i$. In particular, $f^{-1}(e)\subset e$,
so after perturbing $f$ by a homotopy supported on $e$, we may assume
there is an interval $J\subset e$ fixed by $f$ and such that
$f^{-1}(J)=J$.

We now form the topological mapping torus
$X_f\defeq X\times [0,1]/(x,0)\sim (f(x),1)$, whose fundamental group
is $G_\phi$.  In $X_f$, the interval $J$ gives rise to an embedded
annulus, $J\times S^1$, and hence (by Van Kampen's Theorem) to a
splitting of $\pi_1(X_f)=G_\phi$ over an infinite cyclic group:
$G_\phi=H_1\ast_\ZZ H_2$.  Moreover, $H_i$ is precisely the mapping
torus of $\phi|_{G_i}\colon G_i\to G_i$, and is therefore in
$\AC(\vnil)$, by induction.

Let $T$ be the Bass-Serre tree associated to this splitting. It is
not necessarily the case that $G_\phi\curvearrowright T$ is
acylindrical.  
\begin{lemma}\label{lem:EdgeStabilizers}
  If $E,E'$ are two edges in $T$ then
  either $\Stab(E)=\Stab(E')$ or $\Stab(E)\cap\Stab(E')=1$.
\end{lemma}
\begin{proof}
  Consider the standard presentation,
  \begin{equation*}
    G_\Phi=\langle x_1,\ldots,x_n,t\mid tx_it^{-1}=\Phi(x_i)\rangle.
  \end{equation*}
  where $x_1,\ldots, x_n$ is a generating set for $G$.  We recall
  that every element $h\in G_\Phi$ can be written uniquely as
  $h=gt^k$ for some integer $k$ and some element $g\in G$.
  Now the fundamental group of the embedded annulus, $\pi_1(J\times
  S^1)$, is identified in $G_\Phi$ with the stable letter, $t$.
  Evidently the edge stabilizers in $T$ are precisely the conjugates
  of $\langle t\rangle$, $gtg^{-1}=g\Phi(g^{-1})t$. 
  Thus if
  $\Stab(E)\cap\Stab(E')\neq 1$, then for some $n>0$ we
  have $$(g\Phi(g^{-1})t)^n=t^n$$ (the exponents must be equal by
  looking at the homomorphism to $\Z$ that kills the $x_i$'s). By a
  straightforward induction the left-hand side can be written as
  $g\Phi^n(g^{-1})t^n$ so the equation says $\Phi^n(g)=g$. By the
  neatness of $\Phi$ we have $\Phi(g)=g$, so the equation holds for
  $n=1$, i.e. $Stab(E)=Stab(E')$.
\end{proof}

Consider an action $G\curvearrowright T$ on a simplicial tree by
isometries. Recall that a \emph{transverse covering} \cite[Definition
4.6]{G:LimitGroups} of $T$ is a $G$-invariant family $\mathcal{Y}$ of
non-degenerate closed subtrees of $T$ such that any two distinct
subtrees in $\mathcal{Y}$ intersect in at most one point.  A
transverse covering $\mathcal{Y}$ gives rise to a new tree, $S$,
called the \emph{skeleton of $\mathcal{Y}$} as follows: the vertex set
$V(S)=V_0\cup V_1$ where the elements of $V_0$ are in one-to-one
correspondence with elements of $\mathcal{Y}$, and the elements of
$V_1$ are in correspondence with nonempty intersections (necessarily
consisting of a single point) of distinct elements of $\mathcal{Y}$.
Edges are determined by set containment: there is an edge from
$x\in V_1$ to $Y\in V_2$ if $x\in Y$.  The action of $G$ on
$\mathcal{Y}$ determines an action $G\curvearrowright S$.

We use Lemma \ref{lem:EdgeStabilizers} to define a transverse covering
of $T$ whose skeleton will be acylindrical.  For an edge $E$ of $T$,
we let $T_E$ be the forest in $T$ consisting of edges whose stabilizer
is equal to $\Stab(E)$.  That $T_E$ is connected follows from Lemma
\ref{lem:EdgeStabilizers}. We are interested in understanding the
stabilizer $\Stab(T_E)$, so suppose that $E,E'\in T_E$.  Without loss
of generality, assume that $\Stab(E)=\langle t\rangle$ and that $w\in
\Stab(T_E)$ is such that $w^{-1}\cdot E=E'\in T_E$.  As above, we can
write $w$ uniquely as $w=gt^k$.  The definition of $T_E$ provides that
the stabilizer of $E'$ is equal to $\langle t\rangle$.  On the other
hand,
\begin{equation*}
  \Stab(E')=\langle wtw^{-1} \rangle
  = \langle(gt^k) t(t^{-k}g^{-1})\rangle
  =\langle gtg^{-1}\rangle
  =\langle g\Phi(g)^{-1}t\rangle.
\end{equation*}
In particular, $g=\Phi(g)$, and we conclude that
$\Stab(T_E)\simeq \langle t\rangle\times \Fix(\Phi)$.  Neumann proved
\cite{Neumann:FixedGroupOfAnAutomorphism} that for a hyperbolic group
$G$, $\Fix(\Phi)$ is quasiconvex in $G$ and therefore is itself a
hyperbolic group.  Since $\AC(\vnil)$ is closed under taking products,
and hyperbolic groups belong to $\AC(\vnil)$, we conclude that $\Stab(T_E)$
belongs to $\AC(\vnil)$.

The subtrees $\{T_E\}_{E\in T}$ form a transverse covering
$\mathcal{Y}$ of $T$. Let $S$ denote the skeleton of this transverse
cover (this is the {\it tree of cylinders} of \cite{gl:cylinders}). We
observe the following:
\begin{lemma}
  The action $G_\Phi\curvearrowright S$ is acylindrical.
\end{lemma}
\begin{proof}
  Let $v,v'\in V(S)$ be vertices with $d(v,v')\geq 6$ and suppose
  $g\in\Stab(v)\cap\Stab(v')$.  By moving to adjacent vertices if
  necessary, we may assume that $v,v'\in V_1$ so that they are labeled
  by intersections of subtrees in $\mathcal{Y}$ (i.e., points in $T$)
  rather than by trees themselves.  We will again denote by $v$ and
  $v'$ the corresponding points in $T$.  Now $g$ must fix the segment
  in $T$ connecting $v$ to $v'$, and even after moving to adjacent
  vertices we still have $d_S(v,v')\geq 4$.  In particular, there are
  two vertices in $V_0$ on the segment in $S$ between $v$ and $v'$;
  hence the segment connecting $v$ and $v'$ in $T$ contains edges in
  two distinct subtrees $T_E$ and $T_{E'}$.  So
  $g\in \Stab(T_E)\cap\Stab(T_{E'})$ must stabilize two edges with
  different stabilizers and so $g=1$ by Lemma
  \ref{lem:EdgeStabilizers}.
\end{proof}

To conclude our proof of Proposition \ref{prop:PGFJC} we recall that
the vertex stabilizers in $S$ come in two flavors: stabilizers of
vertices in $V_1$ are subgroups of vertex stabilizers in $T$, which
belong to $\AC(\vnil)$ by induction; and stabilizers of vertices in
$V_0$, which are isomorphic to $Fix(\Phi)\times \mathbb{Z}$ and also belong to
$\AC(\vnil)$.  We have thus produced an acylindrical action of
$G_\phi$ on a tree $S$ in which all stabilizers belong to
$\AC(\vnil)$.  Theorem \ref{Knopf} implies that $G_\phi$ belongs to
$\AC(\vnil)$ as well.
\end{proof}


\section{The EG case ($\infty$-ly many ends)}
\label{sec:EG-autos}

The following proposition is all that remains of Theorem
\ref{thm:main-hyperbolic} and is the goal of the present section.

\begin{proposition}\label{prop:EGFJC}
  Let $G$ be a torsion-free hyperbolic group with infinitely many ends
  and let $\Phi\in\Out(G)$ be an automorphism that does not grow
  polynomially. 
  Then $G_\Phi\in\AC(\vnil)$.   
\end{proposition}

It is known from train-track theory that if $||\phi^n(g)||$ does not
grow polynomially then it grows exponentially, but we will not need this.
The result follows immediately from the work above and the following
theorem of Dahmani-Krishna.
\begin{theorem}[{\cite[Theorem 1.1 and 1.2]{DK:RelHypHypAutos}}]\label{prop:RelHypHypAutos}
  If $G$ is a torsion-free hyperbolic group, and $\Phi\in\Aut(G)$,
  then $G_\Phi$ is hyperbolic relative to a family
  $\{P_1,\ldots,P_k\}$ of subgroups, each of which has the form
  $H_i\rtimes_{\psi_i}\mathbb{Z}$ for a finitely generated
  torsion-free hyperbolic group $H_i$ and a polynomially growing
  automorphism $\psi_i\colon H_i\to H_i$.  Moreover, an element
  $h\in G$ grows polynomially under $\Phi$ if and only if it is
  conjugate into some $H_i$.
\end{theorem}

\begin{proof}[Proof of Proposition \ref{prop:EGFJC}]
  Let $G$ and $\Phi$ be as in the statement. Theorem \ref{prop:RelHypHypAutos} shows that $G_\Phi$ is hyperbolic relative to a collection of subgroups, each of which is a mapping torus of a polynomially growing automorphism of a torsion-free hyperbolic group.  By Proposition \ref{prop:PGFJC} (and Proposition \ref{prop:FJC1EndedHyperbolicByCyclic}),
such groups are in $\AC(\vnil)$.  An application of Theorem \ref{Bartels} then shows that $G_\Phi\in\AC(\vnil)$.
\end{proof}

\section{Extensions}
\label{sec:extensions}
In this section, we prove Theorem
\ref{thm:extensions-hyperbolic}.

\newtheorem*{thm:extensions-hyperbolic}{Theorem
  \ref{thm:extensions-hyperbolic}}
\begin{thm:extensions-hyperbolic}
  Let $G$ be a virtually torsion-free hyperbolic group and let
  \begin{equation*}
    1\longrightarrow G\longrightarrow
    H\overset{\pi}{\longrightarrow} Q{\longrightarrow} 1
  \end{equation*}
  be a short exact sequence of groups.  If $Q$ satisfies the
  Farrell-Jones conjecture, then
  so does $H$.
\end{thm:extensions-hyperbolic}

\begin{proof}
  We use the general fact (see \cite[Theorem 2.7]{BFL:jams}) that if
  $Q$ satisfies FJC and for every virtually cyclic subgroup $Z<Q$ the
  preimage $\pi^{-1}(Z)<H$ satisfies FJC, then $H$ satisfies
  FJC. Further, it suffices to show that a finite index subgroup of
  $\pi^{-1}(Z)$ satisfies FJC. The preimage of a finite index cyclic
  subgroup of $Z$ will have the form $G\rtimes \Z$ and will have
  finite index in $\pi^{-1}(Z)$. Theorem \ref{thm:main-hyperbolic}
  implies that $G\rtimes \Z$ satisfies FJC.
\end{proof}

\section{Proof of Theorem \ref{neat}}\label{s:neat}

The proof will
run by induction on the Grushko rank, so we will be
considering the splitting $G=G_1*G_2$ or $G=G'*_1$ given by the
topmost edge in Proposition \ref{CT2}. The basis of the induction is
Corollary \ref{fixed}. To summarize the inductive step, we
will prove the following.

\begin{proposition}\label{step}
  Suppose a torsion-free group $G$ is acting on a tree $T$ with
  trivial edge stabilizers, $\Phi\in Aut(G)$ is an automorphism, and
  $t:T\to T$ is a $\Phi$-equivariant isometry, meaning that
  $t(g(x))=\Phi(g)t(x)$ for any $x\in T, g\in G$. Also assume that $t$
  preserves the $G$-orbits of oriented edges. Suppose further that
  whenever $v$ is a vertex of $T$ such that $t(v)=v$ then $\Phi|V:V\to
  V$ is neat, where $V=Stab(v)$. Then $\Phi$ is neat.
\end{proposition}

We then apply this to the splitting of $G$ given by the topmost edge
in Proposition \ref{CT2}. Proposition \ref{step} is proved in a
sequence of three
lemmas.

\begin{lemma}\label{loxodromic}
  Suppose a lift $\Phi\in Aut(G)$ of $\phi$ is realized as a
  loxodromic $\Phi$-equivariant isometry $t:T\to T$. Then $\Phi$ is
  neat.
\end{lemma}

\begin{proof}
  Suppose $\Phi^n(g)=g$ for some $n>0$ and $g\in G, g\neq 1$. If $g$
  is elliptic in $T$, its fixed point set is a single vertex since
  edge groups are trivial. Thus $t^n$ fixes this vertex contradicting
  the assumption that $t$ is loxodromic. Thus $g$ is loxodromic, and
  its axis is preserved by $t^n$. It follows that $t$ and $g$ have the
  same axis. Thus $g$ and $tgt^{-1}=\Phi(g)$ have the same axis and
  the same (signed) translation length, so $g^{-1}\Phi(g)$ fixes many edges and
  must be trivial.
\end{proof}

Now we assume that $t:T\to T$ realizing $\Phi$ is an elliptic
isometry. The next lemma says that in this case periodic edges
incident to a fixed vertex are fixed.

\begin{lemma}\label{periodic edges are fixed}
  If $t$ fixes a vertex $v$ of $T$, then the incident edges are either
  fixed by $t$, or the $t$-orbit is infinite.
\end{lemma}

\begin{proof}
  Let $V=Stab(v)$ and let $c$ be an edge incident to $v$. Then $t(c)$
  is in the same $V$-orbit as $c$, so we may write $t(c)=a(c)$ for
  some $a\in V$. Suppose now that $t^n(c)=c$. Then $t^n$ fixes the
  whole $t$-orbit of $c$, so $t^n(a(c))=a(c)$. This implies
  $\Phi^n(a)t^n(c)=a(c)$, i.e. $\Phi^n(a)=a$. Since $\Phi|V$ is neat,
  this implies $\Phi(a)=a$. Now we see that
  $t^2(c)=t(a(c))=\Phi(a)t(c)=a^2(c)$ and by induction
  $t^k(c)=a^k(c)$. In paricular $c=t^n(c)=a^n(c)$ so $a^n=1$. Since
  $G$ is torsion-free we have $a=1$ so
  $t(c)=c$.
\end{proof}

The next lemma finishes the proof of Proposition \ref{step} and
Theorem \ref{neat}.

\begin{lemma}
  Suppose $t:T\to T$ is elliptic, $g\in G$, $\Phi^n(g)=g$ for some
  $n>0$. Then $\Phi(g)=g$.
\end{lemma}

\begin{proof}
  First suppose that $g$ is loxodromic in $T$. Then $t^n$ fixes the axis
  of $g$. Lemma \ref{periodic edges are fixed} implies that the
  subtree fixed by $t^n$ is the same as $Fix(t)$, so $t$ also fixes
  the axis of $g$. Thus $g$ and $tgt^{-1}$ have the same axis and the
  same (signed) translation length, so they are equal,
  i.e. $\Phi(g)=g$.

  Now suppose $g$ is elliptic in $T$. Again we have that $t^n$, and
  hence $t$, fixes $Fix(g)$. It follows that there is vertex $v$ of
  $T$ fixed by both $g$ and $t$, and so we have $\Phi(g)=g$ from the
  fact that $\Phi|Stab(v)$ is neat.
\end{proof}

\begin{remark}
  Theorem \ref{neat} holds for all $\phi\in Out(G)$, not only
  polynomially growing ones. This follows from Theorem
  \ref{prop:RelHypHypAutos} and the polynomially growing case since
  any periodic conjugacy class is peripheral.
\end{remark}

\bibliographystyle{alpha}%
\bibliography{bibliography}
\end{document}